\documentclass[12pt]{amsart}
\usepackage {amsxtra,amssymb,amsmath,amsfonts,amscd, cite}

\newtheorem{theorem}{Theorem}[section]
\newtheorem{proposition}[theorem]{Proposition}
\newtheorem{proposition/Definition}[theorem]{Proposition/Definition}

\newtheorem{lemma}[theorem]{Lemma}

\newtheorem{fact}[theorem]{Fact}

\newtheorem{definition}[theorem]{Definition}

\newtheorem{corollary}[theorem]{Corollary}

\theoremstyle{plain}
\numberwithin{equation}{theorem}

\theoremstyle{remark}
\newtheorem{remark}[theorem]{Remark}
\newtheorem{example}[theorem]{Example}

\newcommand{\C}{{\mathbb C}}

\newcommand{\R}{{\mathbb R}}
\newcommand{\Z}{{\mathbb Z}}

\renewcommand{\L}{{\mathcal L}}

\newcommand{\Pl}{{\mathbb P}}

\DeclareMathOperator{\Spec}{Spec}

\DeclareMathOperator{\Eff}{Eff}

\DeclareMathOperator{\Aut}{Aut}
\DeclareMathOperator{\N}{\mathbb{N}}

\DeclareMathOperator{\Pic}{Pic}

\DeclareMathOperator{\Car}{{Car}}
\DeclareMathOperator{\Amp}{{Amp}}
\DeclareMathOperator{\Div}{div}

\newcommand{\lra}{\longrightarrow}
\newcommand{\cO}{\mathcal{O}}

\DeclareMathOperator{\Rat}{Rat}

\title[Polarized dynamics and the Dirichlet property]{Hyperbolic polarized dynamics, pairs of inverse maps and the Dirichlet property}
\author{J. Pineiro}

\begin{document}
\begin{abstract}
 We explore some intersection properties of divisors associated to polarized dynamical systems on algebraic surfaces.  As a consequence, we obtain necessary geometric conditions for the existence of polarizations of hyperbolic type and exhibit compactified divisors associated to automorphisms on K3 surfaces that do not have the Dirichlet property as defined by Moriwaki.
\end{abstract}
\maketitle

\section{Introduction}   Let $K$ be a field, $X$ a geometrically integral, normal projective variety defined over $K$ and $\varphi \colon X \lra X$ a surjective, finite map also defined over $K$.  A situation like this will be called an algebraic dynamical system over the field $K$ and, we will say that the dynamical system $\varphi \colon X \lra X$ is polarized by a divisor $0 \neq E \in \Car(X)$, if there exist $\alpha > 1$ such that $\varphi^* E \sim \alpha E$.  In a similar way a dynamical system on $X$ can be polarized by a real divisor $E \in \Car(X) \otimes \R$.  

Now, for divisors $D_1,D_2,\dots,D_d$ on a normal projective variety $X$ of dimension $d$, we have a multilinear intersection product $(D_1,D_2,\dots,D_d)$, that depends only on the linear equivalence class of the $D_i$.  In particular we have the self-intersection $(D^{d})=(D, \dots , D)$ of a divisor $D$ with itself $d$ times.  Due to linearity, the intersection index is closely related to the polarization property.  Let $X$ be of dimension $d$ and consider the polarized dynamical system $(X,\varphi,E,\alpha)$ on $X$.  From $\varphi^*E \sim \alpha E$ we obtain $\alpha^d (E^{d})= ((\varphi^*E)^{d})= \varphi^*(E^{d}) =
\deg(\varphi) (E^{d}).$ 
As a consequence we have the following two possibilties:
\begin{enumerate}
\item[(1)]  If the polarizing divisor $E$ is ample, then $\deg(\varphi)=\alpha^d$.
\item[(2)]  If $\deg(\varphi) \neq \alpha^d$, then the self-intersection $(E^{d})= 0$.
\end{enumerate}
In the particular case of automorphisms:
\begin{fact} \label{fact} If $\varphi \colon X \lra X$ is an automorphism on $X$ and $(X,\varphi,E,\alpha)$ is a polarized dynamical system, then the self-intersection $(E^{d})= 0$.
\end{fact}
Suppose that we are in dimension $d=2$.  Let $X$ be a smooth projective surface and let $\varphi \colon X \lra X$ be an automorphism on $X$.  A situation like this was studied by Wehler \cite{We} for families of K3 surfaces with an infinite group of automorphism.  Silverman in \cite{Sil-K3} focused in the family $S_{a,b} $ of Wehler $K3$ surfaces obtained as complete intersection of a $(1,1)$-form and a $(2,2)$-form in $\Pl^2 \times \Pl^2$. The family $S_{a,b}$ comes equipped with projections $p_x$ and $p_y$ representing double coverings of
$\Pl_K^2$.  These projections determine rational maps $\sigma_x$ and $\sigma_y$ in
each of the members of the family.  Let us suppose that $\sigma_x$ and $\sigma_y$ are well defined morphisms and put $\varphi=\sigma_y \circ \sigma_x$.  We can find divisors $E_+$ and $E_-$ polarizing the maps $\varphi$ and $\varphi^{-1}$ respectively.  The real divisor $E_+ + E_-$ is ample and, as part of proposition 2.5 in \cite{Sil-K3}, the intersection numbers $(E_+,D)$ and $(E_-,D)$ are positive for every nonzero effective divisor $D$.

A different family of Wehler K3 surfaces is the family $S_c$ studied by Wang \cite{W1} , Billard \cite{Bi}  and Baragar \cite{B1}, \cite{B2}.  This new family is defined by a $(2,2,2)$-form in $\Pl^1 \times \Pl^1 \times \Pl^1$ and shares many of the properties of $S_{a,b}$.  We can find three involutions $\sigma_{2,3}, \sigma_{1,3}, \sigma_{1,2} \colon S_c \lra S_c$, that are well defined automorphisms on generic members of the family.  In \cite{Bi}, for example, the author introduces divisors $E_1$ and $E_2$ polarizing the maps $\varphi=\sigma_1 \sigma_3 \sigma_2$ and $\varphi^{-1}=\sigma_2 \sigma_3 \sigma_1$ respectively.  Again $E_1+E_2$ is ample and the intersection numbers $(E_1,D)$ and $(E_2,D)$ are positive for every nonzero effective divisor $D$.

The results obtained for the dynamical system $(X, \varphi, E, \alpha)$ associated to the action of an automorphism $\varphi$ on a surface $X$ depend on the existence of a polarized dynamics $(X, \varphi^{-1}, E', \alpha)$ for the inverse map.  Similar results can be obtained for the dynamical system $(X, \varphi, E, \alpha)$ of any self-map $\varphi \colon X \lra X$ provided that the polarization is ``hypebolic" in the sense that $\alpha^{-1}$ is also an eigenvalue of the linear map $\varphi^* \colon \Pic(X) \otimes \R \lra \Pic(X) \otimes \R$. If $\varphi \in \Aut(X)$, the existence of a hyperbolic polarization for $\varphi$, with say ${\varphi^{-1}}^* E' \sim \alpha^{-1} E'$, is equivalent to the existence of a pair of polarized dynamics $(X, \varphi, E, \alpha)$ and $(X, \varphi^{-1}, E', \alpha)$, for the map $\varphi$ and its inverse.  To be able to obtain positive intersection with effective divisors we restrict ourselves to the ample case, i.e., when $E+E'$ is ample.

\begin{proposition} \label{intro_1} Let $(X, \varphi, E, \alpha)$ be an ample hyperbolic polarized system on the smooth surface $X$ over the field $K$.  Then, for every nonzero effective divisor $0 \neq D \in \Eff(X)_\R$, we have $(E,D)>0$.  In the case $\varphi \in \Aut(X)$ we will have two polarized dynamics $(X, \varphi, E, \alpha)$, $(X, \varphi^{-1}, E', \alpha)$ and the intersection numbers $(E,D)$ and $(E',D)$ are both positive. \end{proposition}

\begin{example} In section 2.4, we work with the family $S_c$ of K3 surfaces and present three maps $\tau_1,\tau_2$ and $\tau_3$ with ample hyperbolic polarization given by pairs $(S, \varphi, E_i, \alpha)$ and $(S, \varphi^{-1}, E_j, \alpha)$.  Many arithmetic and geometric results concerning rational points were studied in \cite{Bi} in the system $\tau_2=[(S, \sigma_1 \sigma_3 \sigma_2, E_{1}, \beta^3) \, , \,     (S, \sigma_2 \sigma_3 \sigma_1 , E_{2}, \beta^3)]$, where $\beta=\frac{3+\sqrt{5}}{2}$. \end{example}

In section 3 we deal with the arithmetic case, putting canonical metrics on polarizing divisors $E$ associated to systems $(X, \varphi, E, \alpha)$ defined over number fields $K$.  The possibility of putting canonical metrics relative to the map $\varphi \colon X \lra X$ on the associated line bundle $\L=\cO(E)$ at every place $v$ of $K$, was introduced by S. Zhang in \cite{Zhang3}.  In later presentations authors, like for example Chambert-Loir \cite{Cha}, considered analytic spaces $\pi_v \colon X_v^{an} \lra X$ associated to $X$ at each place $v$ of $K$ and metrics on the analytifications $\L_v=\pi_v^*\cO(E)$ for each place $v$ of $K$.  For a real divisor $E$ polarizing a dynamical system $\varphi \colon X \lra X$, there exist a canonical way (Proposition 2.1.1 and Proposition 2.2.1 in \cite{Chen1}) to put a metric $\|.\|_{\varphi,v}$, associated to the map $\varphi$, on the analytification $\L_v=\pi_v^* \L$ of the line bundle $\L=\cO(E)$.  The collection of metrics $(\|.\|_{\varphi,v})_v$ is adelic in the sense that it is induced by the same integral model of $X$ for almost all finite places $v$ (See remark \ref{Canonical metric is adelic} in section 3).  The metric $(\|.\|_{\varphi,v})_v$ is called the canonical metric on $E$ and the metrized divisor $\bar{E}=(E,\|.\|_{\varphi, v})$ is called the canonical compactification of $E$ with respect to the map $\varphi$.  We refer to sections 2 and 3 of \cite{Chen1} for the proofs of the existence and properties of the canonical metric.  For example, in section 3 of  \cite{Chen1}, canonical compactifications are studied in relation to a higher dimensional analogue of Dirichlet unit's theorem.  Suppose that we extend our notion of effective to metrized divisors $\bar{D}$, imposing the extra condition that the canonical section $s_D$ of our effective divisor $D$, has norm $\|s_D\|_{v,\sup} \leq 1$ at every place.  The Dirichlet property for metrized divisors on arithmetic varieties, first considered by Moriwaki in \cite{M1}, can be stated as follows: 

\begin{definition} We say that the adelic metrized divisor $\bar{E}$ satisfies the Dirichlet property if $\bar{E}$ is $\R$-linearly equivalent to an effective metrized divisor. \end{definition}

As a combination of proposition \ref{intro_1} and fact \ref{fact} we are able to get a general result for compactified divisors associated to ample hyperbolic polarizations.  At the same time, we obtain examples of adelic metrized divisors without the Dirichlet property in families of K3 surfaces.

\begin{proposition}  Let $(X,\varphi, E, \alpha)$ be an ample hyperbolic polarized system in the smooth surface $X$, with $\deg(\varphi) \neq \alpha^2$.  Then, the canonical compactification $\bar{E}$ do not satisfy the Dirichlet property.  Also, for a pair of polarized systems $(X, \varphi, E, \alpha)$, $(X, \varphi^{-1}, E', \alpha)$ with $E+E'$ ample, the canonically compactified divisors $\bar{E}$ and $\bar{E}'$ do not satisfy the Dirichlet property.  
\end{proposition}

\begin{corollary} In the family $S_{a,b}$, the compactified divisors $\bar{E}^{+}, \bar{E}^{-}$ do not have the Dirichlet property.  In the family $S_c$, the compactified divisors $ \bar{E}_{i}$ do not have the Dirichlet property.
\end{corollary}


\section{The geometric case} In this section we study some consequences of the polarization property for dynamical systems on smooth projective surfaces.  We start by introducing ample and effective divisors with real coefficients.

\subsection{Real divisors} Let $X$ be a normal projective variety over a field $K$.  Suppose that $\Car(X)$ represents the group of Cartier divisors on $X$.  A Cartier $\R$-divisor on $X$ is an element of $\Car(X)_\R=\Car(X) \otimes \R$.  The injective map from the multiplicative group $\Rat(X)^\ast$ of rational functions on $X$ into $\Car(X)$ extends to a map $\Rat(X)^\ast_{\R}=\Rat(X)^\times \otimes \R \hookrightarrow \Car(X)_{\R}$.  The principal Cartier divisor associated to $f \in \Rat(X)^\ast_{\R}$ will be denoted by $(f)$.  We will refer to Cartier divisor simply as divisors on $X$.

\begin{definition}
Two $\R$-divisors $D_1$ and $D_2$ are called $\R$-linearly equivalent, denoted $D_1 \sim D_2$, if they differ by a principal real divisor, this is, if there exist $f \in \Rat(X)^\times_\R$ such that $D_1-D_2=(f)$.  
\end{definition}

\begin{remark} Let $D_1, D_2, \dots, D_n \in \Car(X)_\R$, the intersection number $(D_1,D_2,\dots,D_d)$ can be defined by linearity and the result depends only on the $\R$-linear equivalence class of the $D_i$.
\end{remark}

\begin{definition}
An $R$-divisor $D$ is said to be ample (resp. effective) if $D$ can be written as $D=\sum a_i D_i$ with $a_i >0$ and the $D_i$ are ample (resp. effective).  We denote by $\Eff(X)_\R$ and $\Amp(X)_\R$ the cones of effective and ample real divisors on $X$ respectively.
\end{definition}

\subsection{Polarized dynamical systems} Let $X$ be a projective, normal, geometrically integral algebraic variety defined over a field $K$ and $\varphi \colon X \lra X$ a finite, surjective self-map of $X$ also defined over $K$.
 Suppose that $E$ is a nonzero $\R$-divisor on $X$ and for some real number $\alpha >1$, we have the linear equivalence $\varphi^*E \sim \alpha E$.  This situation will be called a polarized dynamical system $(X,\varphi,E,\alpha)$ on $X$ defined over the field $K$.
 
 \begin{remark} \label{lambda to the n}
Consider a normal projective variety $X$ of dimension $d$ and a polarized dynamical system $(X,\varphi,E,\alpha)$ defined over $K$.  Let us denote by $(E^{d})$ the self-intersection of $E$ with itself $n$ times.  As long as $\deg(\varphi)  \neq   \alpha^d$, the identity $$\alpha^d (E^{d})= ((\varphi^*E)^{d})= \varphi^*(E^{d}) =
\deg(\varphi) (E^{d}),$$ gives self-intersection $(E^{d})= 0$.
\end{remark}

\subsection{Dynamics on the family $S_{a,b}$ of K3 surfaces} Consider the family of $S_{a,b} \subset \Pl^2 \times
\Pl^2$ studied by Silverman in \cite{Sil-K3}.  The family $S_{a,b}$ is determined by the following two equations with coefficients in the field $K$,
$$\sum^3_{i,j=1}a_{i,j}x_i y_j=\sum^3_{i,j,k,l=1}b_{i,j,k,l}x_ix_k y_j y_l =0.$$
We can prove \cite{We} that such equations determine K3 surfaces.  The projections $p_x$ and $p_y$ represent double coverings of
$\Pl_K^2$ and determine rational maps $\sigma_x$ and $\sigma_y$ in
each of the members of the family.  Suppose that $\sigma_x, \sigma_y$ are morphisms and we have denoted the pull-back divisors by $D_n^x=p_x^* \{x_n=0\}$ and $D^y_m=p_y^*\{y_m=0\}$ respectively.  The geometry of the family can be used to determine the eigenvalues of $(\sigma_y \circ
\sigma_x)^{\pm 1*}$ in the subspace of $\Car(S_{a,b})_\R \otimes \R$ generated by the divisor classes of $D_n^x$ and $D_m^y$.  Indeed for $\varphi=\sigma_y \circ \sigma_x$ and          $\beta=2 + \sqrt{3}$, the real divisors 
$$E^+=E^+_{mn}=\beta D_x^n-D_y^m, \, \, \qquad E^-=E^-_{mn}=-D_x^n+\beta D_y^m$$ will satisfy the two identities
$$\varphi^{*} E^{+} \sim \beta^2 E^{+}  \text{ and } (\varphi^{-1})^{*} E^{-} \sim \beta^2 E^{-} .$$
For values $(a,b)$ of the parameters such that $S_{a,b}$ is smooth, the rational maps $\varphi=\varphi^+=\sigma_y \circ \sigma_x$ and $\varphi^{-1}=\varphi^-=\sigma_x \circ \sigma_y$ define automorphisms $\varphi^+,\varphi^- \colon S_{a,b} \lra S_{a,b}$ polarized by the divisors $E^+,E^-$ respectively.  We observe that the divisor $E^+ + E^- = (1+\sqrt{3}) (D_1 + D_2)$ is an ample $\R$-divisor on $S_{a,b}$.

\subsection{K3 surfaces in $\Pl^1 \times \Pl^1 \times \Pl^1$} Consider a smooth projective variety $S$ defined by a $(2,2,2)$ form in $\Pl^1 \times \Pl^1 \times \Pl^1$.  For $(x,y,z) \in \Pl^1 \times \Pl^1 \times \Pl^1$, the surface $S=S_c$ can be viewed as zero locus of the polynomial
$$  F(x,y,z)= \sum_{i_l + j_l=2, \, l=0,1,2} c_{i_1,i_2,i_3,j_1,j_2,j_3} x_0^{i_1} x_1^{j_1}y_0^{i_2} y_1^{j_2}z_0^{i_3} z_1^{j_3},$$
where the coefficients $c_{i_1,i_2,i_3,j_1,j_2,j_3} $ belong to a field $K$.  Again by the results in \cite{We}, we have a K3 surface.  If we write $F(x,y,z)=x_0^2 F^x_0(y,z) +x_0x_1 F^x_1(y,z)+ x_1^2F^x_2(y,z) $, the equation $F(X,P_2, P_3)=0$ will have exactly two solutions as long as $(P_2,P_3)$ is not a solution of the system of equations $F^x_1=F^x_2=F^x_3=0$.  With such conditions, for every point $P=(P_1,P_2,P_3) \in S(K)$, we can find $P'=(P_1',P_2,P_3)$ as the other solution of $F(X,P_2,P_3)$.  In this way we have defined a rational map $\sigma_1=\sigma_{2,3} \colon S \lra S$.  In similar way we can define rational maps $\sigma_2=\sigma_{1,3} \colon S \lra S$ and $\sigma_3=\sigma_{1,2} \colon S \lra S$.  For values of the parameters $(c)=(c_{i_1,i_2,i_3,j_1,j_2,j_3} )$ such that the three sets of curves $\{F^x_0,F^x_1, F^x_2\}, \{F^y_0,F^y_1, F^y_2\}, \{F^z_0,F^z_1, F^z_2\}$ have no common points, the maps $\sigma_1, \sigma_2$ and $\sigma_3$ are well defined automorphisms of the surface $S$.  We call this case, the generic case, and work from now on with surfaces of the family $S$ satisfying these conditions.  Let $\{t_0\}$ be a point in $\Pl^1_K$ and $p^i \colon S \lra \Pl^1$ the projection onto the $i$-th component.  Let us denote by $D_i$ for $i=1,2,3$ the ample divisor $D_i= (p^i)^*[t_0]$ in $S$.   For a generic surface $S$, the Picard number is three and we will work with the basis $\mathcal{D}=\{D_1, D_2,D_3\}$ of $\Car(S)_\R$.   The intersection matrix $J$ as well as the action of the maps $\sigma_i$ on elements of the basis were studied in \cite{W1}, \cite{Bi}, \cite{B1} and \cite{B2}.  Suppose that we put $\beta=\frac{3+\sqrt{5}}{2}$, $a=\frac{-3+\sqrt{5}}{2}$ and $b=\frac{-1+\sqrt{5}}{2}$, and name the divisors $$E_1=[1,a,b]_\mathcal{D} \qquad E_5=[b,1,a]_\mathcal{D} \qquad E_3=[a,b,1]_\mathcal{D}$$
$$E_4=[1,b,a]_\mathcal{D} \qquad E_2=[a,1,b]_\mathcal{D} \qquad E_6=[b,a,1]_\mathcal{D}$$



 The maps $\sigma_{i,j,k}^\ast = (\sigma_i \circ \sigma_j \circ \sigma_k)^\ast=\sigma_k^\ast \circ \sigma_j^\ast \circ \sigma_i^\ast$ share the same eigenvalues $\{\beta^3,\beta^{-3},-1\}$ for $\{i,j,k\}=\{1,2,3\}$.  The eigenvectors associated to the different eigenvalues $\lambda$ for $\sigma_{i,j,k}^\ast$ can be computed as presented in the table below.

\begin{center}
\begin{tabular}{|c|c|c|c|} \hline
Morphism & $\lambda=\beta^3$ & $\lambda=\beta^{-3}$ & $\lambda=-1$ \\ \hline
$\sigma_{3,2,1}^\ast=\sigma_1^\ast \circ \sigma_2^\ast \circ \sigma_3^\ast$ & $E_{3}=[a,b,1]$ & $E_4=[1,b,a]$ & $[1,-3,1]$ \\ \hline
$\sigma_{1,3,2}^\ast=\sigma_2^\ast \circ \sigma_3^\ast \circ \sigma_1^\ast$ & $E_{1}=[1,a,b]$ & $E_{2}=[a,1,b]$ & $[1,1,-3]$ \\ \hline
$\sigma_{2,1.3}^\ast=\sigma_3^\ast \circ \sigma_1^\ast \circ \sigma_2^\ast$ & $E_{5}=[b,1, a]$ & $E_{6}=[b,a,1]$ & $[-3,1,1]$ \\ \hline
$\sigma_{3,1,2}^\ast=\sigma_2^\ast \circ \sigma_1^\ast \circ \sigma_3^\ast$ & $E_{6} =[b,a,1]$ & $E_{5}=[b,1,a]$ & $[-3,1,1]$ \\ \hline
$\sigma_{2,3,1}^\ast=\sigma_1^\ast \circ \sigma_3^\ast \circ \sigma_2^\ast$ & $E_{2}=[a,1,b]$ & $E_{1}=[1,a,b]$ & $[1,1,-3]$ \\ \hline
$\sigma_{1,2,3}^\ast=\sigma_3^\ast \circ \sigma_2^\ast \circ \sigma_1^\ast$ & $E_{4}=[1, b, a]$ & $E_{3}=[a,b,1]$ & $[1,-3,1]$ \\ \hline
\end{tabular}
\end{center}  \vspace{2mm}

 For instance, we have three pairs of polarized dynamical systems on $S$ given by a map $\varphi \colon S \lra S$ and its inverse $\varphi^{-1} \colon S \lra S$, namely the pairs:
$$\tau_1=[(S, \sigma_{3,2,1}, E_{3}, \beta^3) \, , \,   (S, \sigma_{1,2,3}, E_{4}, \beta^3)],$$
$$\tau_2=[(S, \sigma_{1,3,2}, E_{1}, \beta^3) \, , \,     (S, \sigma_{2,3,1} , E_{2}, \beta^3)],$$
$$\tau_3=[(S, \sigma_{2,1,3}, E_{5}, \beta^3) \, , \,    (S, \sigma_{3,1,2}, E_{6}, \beta^3)].$$ 
We can check that the divisors $E_3 + E_4$, $E_1+E_2$ and $E_5 + E_6$ are ample real divisors divisors in $\Car(S)_\R$.  The pair of maps $\tau_2$ was work out in full details in propositions 1.5 and 1.8 of \cite{Bi}.

\subsection{Hyperbolic polarizations and automorphisms on algebraic surfaces} In this part we obtain some general intersection properties of real divisors associated to a special type of polarization: the hyperbolic polarizations.  

\begin{definition} A polarized dynamics  $(X, \varphi, E, \alpha)$ defined over $K$ on $X$ will be called hyperbolic polarized if there exist a real divisor $0 \neq E' \in \Car(X)_\R$ such that $\varphi^* E' \sim \frac{1}{\alpha} E'$.  Furthermore, if $E+E'$ is ample, we will say that the algebraic dynamical system $(X,\varphi, E,\alpha)$ is ample hyperbolic polarized.
\end{definition}

\begin{remark} Let $X$ be a normal projective surface defined over a field $K$ and  $(X, \varphi, E, \alpha)$ a polarized dynamical system associated to an automorphism $\varphi \in Aut(X)$ over $K$.   The system  $(X, \varphi, E, \alpha)$ is hyperbolic polarized with divisor $E'$ if and only if we can find a pair of polarized dynamical systems   $(X, \varphi, E, \alpha)$ and $ (X, \varphi^{-1}, E', \alpha)$ associated to the map $\varphi \colon X \lra X$ and its inverse $\varphi^{-1} \colon X \lra X$.  Again in this situation we will say that $(X, \varphi, E, \alpha)$ is ample hyperbolic polarized if $E+E'$ is ample.
\end{remark}

\begin{lemma} \label{fund_lemma} Let $X$ be smooth projective surface, $0 \neq D \in \Eff(X)_\R$ a nonzero real effective divisor on $X$, $E \in \Amp(X)_\R$ an ample real divisor on $X$ and $\varphi \colon X \lra X$ a self-map.  Then, there exist $c >0$ such that $((\varphi^n)^*E,D) > c$ for all natural numbers $n$.
\end{lemma}
\begin{proof} By definition of what it means to be effective and linearity, is enough to do the proof for $0 \neq D \in \Eff(X)$, an integral effective divisor.  Let $E$ be a finite sum $E = \sum \alpha_i E_i$ with $E_i \in \Amp(X)$ ample and $\alpha_i \in \R$.  For some $N \in \N$ we will have $N E_i$ is very ample for all $i$ and $((\varphi^n)^*(NE_i),D) \geq 1$ for all natural numbers $n$.  Then
$$((\varphi^n)^*E,D)=\frac{1}{N} \sum_i \alpha_i ((\varphi^n)^*(N E_i),D) \geq \sum_i \frac{\alpha_i}{N}=c > 0,$$ which is the inequality we want to prove.\end{proof}

\begin{proposition} \label{Necessary condition for effective} Let $(X, \varphi, E, \alpha)$ be polarized system on the smooth surface $X$ over the field $K$.  Suppose that $(X, \varphi, E, \alpha)$ is ample hyperbolic polarized.  Then, for every nonzero effective divisor $0 \neq D \in \Eff(X)_\R$, we have $(E,D)>0$.
\end{proposition}
\begin{proof} The definition of hyperbolic polarized provide us with a divisor $E' \in \Car(X)_\R$ such that $\varphi^* E = \alpha E$ and $\varphi^* E' = \alpha^{-1} E'$.  The condition of ample forces $E+E' \in \Amp(X)_\R$. Now, take a real effective divisor $0 \neq D \in \Eff(X)_\R$, an integer $n > 0$ and compute:
$$\alpha^n (E,D) + \alpha^{-n} (E',D) = ((\varphi^n)^*(E+E'), D) > c > 0,$$
where the last two inequalities are the result of applying \ref{fund_lemma} to the ample divisor $E + E'$ and the effective divisor $D$.  Since the inequality holds for any $n \in \N$, we must have $(E,D)>0$.\end{proof}

\begin{corollary} \label{self-intersection on a surface}  Let $(X, \varphi, E, \alpha)$ be ample hyperbolic polarized system on the smooth surface $X$ over the field $K$.  Then, if $\deg(\varphi) \neq \alpha^2$, the divisor $E$ is not $\R$-linearly equivalent to an effective divisors in $X$.
\end{corollary}
\begin{proof} The condition $\deg(\varphi) \neq \alpha^2$ for the polarized dynamical system  $(X, \varphi, E, \alpha)$ guarantees $(E,E)=0$ after remark \ref{lambda to the n} (or fact \ref{fact} in the introduction).   As the self-intersection depends only on the $\R$-linear equivalence class of $E$, the result is an application of  proposition \ref{Necessary condition for effective} to $D=E$. \end{proof}

\begin{corollary} \label{The_case-of_Aut} Let $X$ be a smooth projective surface and $\varphi \in \Aut(X)$ such that we have a hyperbolic polarization given by two polarized dynamical systems $(X, \varphi, E, \alpha)$ and $(X, \varphi^{-1}, E', \alpha)$, with $E+E' $ is ample.  Then, for every real nonzero effective divisor $D$, the intersection numbers $(E,D)$ and $(E',D)$ are positive.  In particular, the divisors $E,E'$ can not be $\R$-linearly equivalent to effective divisors.
\end{corollary}
\begin{proof}  We already now $(E,D) > 0$ by \ref{Necessary condition for effective}.  Now, consider the inequality 
$$\alpha^n (E,D) + \alpha^{-n} (E',D) = ((\varphi^n)^*(E+E'), D) > c > 0$$
of proposition \ref{Necessary condition for effective} for all $n \in \Z$.  This gives $(E',D) > 0$ as well.
\end{proof}

The following result is the second part of proposition 2.1.1  in \cite{Zhang2} in case $X$ is projective and we have an ample hyperbolic polarized dynamical system $(X, \varphi, E, \alpha)$, with $E$ not ample.

\begin{corollary} \label{Dynamic_condition_on_geometry}  Let $(X, \varphi, E, \alpha)$ be ample hyperbolic polarized system on the smooth surface $X$ over the field $K$ associate to an \'{e}tale map $\varphi \colon X \lra X$ such that $\deg(\varphi) \neq \alpha$.  Let us denote by $K_X$ the canonical divisor of $X$.   Then, for any real number $m$, the divisor $mK_X$ is either zero or not effective on $X$.  In particular, the Kodaira dimension $\kappa(X) \leq 0$.
\end{corollary}
\begin{proof} Since the map $\varphi \colon X \lra X$ is \'{e}tale, we have the linear equivalence $\varphi^* K_X \sim K_X$ and $$\alpha(mK_X,E)=(\varphi^*mK_X, \varphi^*E)=\deg(\varphi) (mK_X,E).$$   The condition  $\deg(\varphi) \neq \alpha$ gives $(mK_X,E)=0$ for all numbers $m$.   We can use proposition  \ref{Necessary condition for effective} to obtain that the divisor $mK_X$ is either zero or not effective for $m \in \R$.
\end{proof}

\begin{example} \label{pair of maps for S_ab} For the family $S_{a,b}$ of Wehler K3 surfaces discussed in subsection 2.3, we have ample hyperbolic polarizations associated to the pair of  polarized dynamics $(S_{a,b}, \varphi^+,E^+,\beta^2)$ and $(S_{a,b}, \varphi^-,E^-,\beta^2)$, for $\beta=2+\sqrt{3}$ and $E^+ + E^-= (1+\sqrt{3}) (D_1 + D_2)$.  By corollary \ref{The_case-of_Aut}, the divisors $E^+,E^-$ are not effective on $S_{a,b}$.\end{example}

\begin{example} \label{pair of maps on S_c}  Let $S_c$ be the family of K3 surfaces in $\Pl^1 \times \Pl^1 \times \Pl^1$ discussed in subsection 2.4.  For $\beta=\frac{3+\sqrt{5}}{2}$, the pairs $\tau_1,\tau_2$ and $\tau_3$ of inverse maps, given by:
$$\tau_1=[(S, \sigma_{3,2,1}, E_{3}, \beta^3) \, , \,   (S, \sigma_{1,2,3}, E_{4}, \beta^3)],$$
$$\tau_2=[(S, \sigma_{1,3,2}, E_{1}, \beta^3) \, , \,     (S, \sigma_{2,3,1} , E_{2}, \beta^3)],$$
$$\tau_3=[(S, \sigma_{2,1,3}, E_{5}, \beta^3) \, , \,    (S, \sigma_{3,1,2}, E_{6}, \beta^3)].$$  
provide us with ample hyperbolic polarizations in $S_c$.  We can apply corollary \ref{The_case-of_Aut} to obtain that $E_i$ are not real effective divisors on $S_c$.
\end{example}

\section{The arithmetic case} In this section we work with geometrically integral, normal projective varieties $X$ defined over a number field $K$ and put metics on divisors or better in their associated line bundles.  We will denote by $\cO_K$ the ring of integers of $K$ and by $\mathcal{M}_K$ the set of places of $K$.

\begin{definition}
Let $X$ be a normal, projective variety of dimension $d$ defined over a number field $K$.  For each place $v \in \mathcal{M}_K$, we associate to $X$ the $v$-adic analytic space $X_v^{an}$.  The association $X \rightsquigarrow X_v^{an}$ is functorial and we have continuous maps $\pi_v \colon X_v^{an} \lra X$.  The analytic spaces $X_v^{an}$ can be described as follows:
  \begin{enumerate}
  \item[(1)] In the Archimedean case,  $\sigma=v \colon K \hookrightarrow \C$, we put  $K_v= K \otimes_\sigma \C$ and the analytic space $X_v$ is just $X_v^{an}=X \times^\sigma_{\Spec(K)} \Spec(K_v)$.
  \item[(2)] In the non-Archimedean case the analytification  $X_v^{an}$ of $X$ is the Berkovich analytic space associated to $X$ over the completion $K_v$ of $K$ at the place $v$.  Berkovich analytic spaces, introduced in \cite{Be}, are locally compact and locally arc-connected.  We refer to section 1.2 and section 1.3 of \cite{B3} for a summary of their properties.  
  \end{enumerate}
\end{definition}
 
 \begin{definition}  Let $X$ be a variety over the number field $K$ and $\L$ a line bundle on $X$.  A metric $\|.\|$ on $\L$ is a collection of metrics $(\|.\|_v)_{v \in \mathcal{M}_K}$, where $\|.\|_v$ is a metric on the analytification $\L_v^{an}=\pi_v^*L$, for every place $v \in \mathcal{M}_K$.  The Green function $g_v$ associated to the metric $\|.\|$ is defined at each place $v$ of $K$ as $g_v(x)=-\log\|s(x)\|_v$, where $s$ is a section of $\L$.
\end{definition}

 A metrized $\R$-divisor $\bar{D}$ will be a pair $(D,\|.\|)$, where $D$ is Cartier $\R$-divisor on $X$ and $\|.\|$ is a metric on the associated line bundle $\cO(D)$.  For instance, for $f \in \Rat(X)^\ast_{\R}$, the principal $\R$-divisor $(f)$ naturally defines a metrized $\R$-divisor $\widehat{(f)}$ when we put, at each place $v$ of $K$, the absolute value of the associate section $|f^{-1}(x)|_v=1$.

\begin{definition} Let $(X,\varphi,E,\alpha)$ be a polarized dynamical system on $X$ defined over $K$, in such a way that $\varphi^* E = \alpha E + (f)$ for some $f \in \Rat(X)^\ast_{\R}$.  Let $\|.\|$ be a metric on $E$.  For every place $v$ we have a continuous function $\lambda_{v,(E,\|.\|)} \colon X_v^{an} \lra K_v$ such that
$$\varphi^*\|.\|_v=|f|_v\|.\|_v^\alpha \lambda_{v,(E,\|.\|)}.$$ 
The function $\lambda_{v,(E,\|.\|_\varphi)} \equiv 1$, for almost all places $v$.  The canonical metric associated to the map $\varphi$ is the unique metric $\|.\|_\varphi$ on $E$ satisfying $\lambda_{v,(E,\|.\|_\varphi)} \equiv 1$ for all places $v$.  In this sense we have
$$\varphi^*(E,\|.\|_\varphi)= \alpha (E,\|.\|_\varphi) + \widehat{(f)} .$$
\end{definition}
\begin{remark}\label {Canonical metric is adelic} The existence of such metric is proved in section 2 of \cite{Chen1}.  For each place $v \in \mathcal{M}_K$, it is possible to define a metrized $\R$-divisor $\bar{E}_v$ on $X_v=X \times \Spec(K_v)$ such that $\lambda_{\bar{E}_v}=1$.  Next, we can find a model $(\mathcal{X}_\mathcal{U}, E_\mathcal{U})$ of $(X, E)$ over a Zariski open $\mathcal{U}$ of $\Spec(\cO_K)$ such that the map $\varphi$ extends to a map $\varphi_\mathcal{U} \colon \mathcal{X}_\mathcal{U} \lra \mathcal{X}_\mathcal{U}$.  The uniqueness of the metric $\|.\|_v$ ensures that $g_v$ is induced by the model $(\mathcal{X}_\mathcal{U}, E_\mathcal{U})$ at every $v \in \mathcal{U}$, as described in section 1.3 of \cite{B3}. 
\end{remark}

\begin{definition}   Let $X$ be a variety over the number field $K$ and $\L$ a line bundle on $X$.  A metric $\|.\|$ on $\L$ will be called quasi-algebraic or adelic if it is induced by the same model $(\tilde{X},\tilde{\L})$ of $(X,\L^e)$, for almost all finite places $v$ of $K$.
\end{definition}

 As a result of the discussion in remark \ref{Canonical metric is adelic} the canonical metric associated to a polarized dynamics  $(X,\varphi,E,\alpha)$ is quasi-algebraic. The adelic metrized $\R$-divisor $\bar{E}=(E,\|.\|_\varphi)$ will be called the canonical compactification of the divisor $E$ with respect to the system $(X,\varphi,E,\alpha)$. 

\begin{example} \label{canonical compactification on K3 examples} We denote by $\bar{E}^+$ and $\bar{E}^-$ the canonical compactification of $E^+,E^-$ with respect to the dynamical systems $(S_{a,b}, \varphi^+,E^+,\alpha^2)$ and $(S_{a,b}, \varphi^-,E^-,\alpha^2)$ on the family $S_{a,b}$ of K3 surfaces.  In the same way we denote by $\bar{E}_i$ the canonical compactification of $E_i$ with respect to the pairs of polarized dynamics of example \ref{pair of maps on S_c} on the family $S_c$.
\end{example}

Now we discuss the notion of effectivity for metrized divisors.  Suppose that $s=fs_D$ is a non-zero global section of $\cO(D)$.  Since the divisor $(s)=(f) + D$ is effective, the function $\|s\|_v \colon X_v^{an} \setminus |\Div(s)| \lra \R$ can be extended to a continuous function on $X_v^{an}$, and we put for each place $v \in \mathcal{M}_K$,
$$\|s\|_{v, sup} = \sup_{P \in X_v^{an}} \|s(P)\|_v.$$

\begin{definition} \label{effective+Dirichlet} The arithmetic $\R$-divisor $\bar{D}$ is effective ($\bar{D} \succeq  0$) if $D$ is effective and the canonical section $s_D$ satisfies $\|s_D\|_{v, \sup} \leq 1$ for all places $v \in \mathcal{M}_K$.
We say that the adelic metrized $\R$-divisor $\bar{D}$ satisfies the Dirichlet property if there exist an $\R$-rational function $f \in \Rat(X)_\R^\ast$ such that $\bar{D} + \widehat{(f)} \succeq  0$.   \end{definition}

 The classical Dirichlet unit theorem on $\cO_K$ reflects the fact that on $X=\Spec(\cO_K)$, an arithmetic divisor with non-negative degree satisfies the Dirichlet property.  The study of higher dimensional analogues was introduced by Moriwaki in \cite{M1} and continued in \cite{Chen1}  and \cite{Chen2}.  Positive and negative answers has been given in some cases, while working with canonically compactified divisors.
 
 \begin{proposition} Let $(X,\varphi, E, \alpha)$ be an ample hyperbolic polarized system in the smooth surface $X$, with $\deg(\varphi) \neq \alpha^2$.  Then, the canonical compactification $\bar{E}$ do not satisfy the Dirichlet property. \end{proposition}
 \begin{proof} If $\bar{E}+\widehat{(f)} =\bar{D} \succeq 0$, then in particular $D$ will be effective which contradicts corollary \ref{self-intersection on a surface}.  \end{proof}
 
 The same proof works for polarizing divisors in pairs of automorphisms on surfaces.  In this case we can prove:

\begin{proposition} For an ample hyperbolic polarized dynamics given by a pair of systems $(X, \varphi, E, \alpha)$ and $(X, \varphi^{-1}, E', \alpha)$, the canonically compactified divisors $\bar{E}, \bar{E}'$ do not satisfy the Dirichlet property.  
\end{proposition}

\begin{corollary} In the family $S_{a,b}$, the compactified divisors $\bar{E}^{+}, \bar{E}^{-}$ do not have the Dirichlet property.  In the family $S_c$, the compactified divisors $ \bar{E}_{i}$ do not have the Dirichlet property.
\end{corollary}

\subsection{Arithmetic degree of compactified divisors} Suppose that $X$ is $d$-dimensional, projective, and geometrically integral normal variety over a number field $K$.  To be able to present the notion of arithmetic degree, we are going to consider a special type of adelic metrized divisors $\bar{D}$ on $X$: the relatively nef divisors.  We will say that the adelic $\R$-divisor $\bar{D}=(D,\|.\|)$ is relative nef if $D$ is a nef divisor and the associated green function $g_v$ is of $(C^0 \cup PSH)$-type for every place $v$ as defined in section 2.1 of \cite{M2}.   

\begin{definition} 
 A divisor $\bar{D}$ is integrable if it can written as the difference of two relatively nef adelic divisors.
\end{definition}

\begin{example} Let  $(X, \varphi, E, \alpha)$ be a polarized dynamical system on $X$ with $E$ ample.  As a consequence of lemma 3.1 in \cite{Chen1}, the canonically compactified divisor $\bar{E}$ is relatively nef. \end{example}
We can have for the integrable divisor $\bar{D}$ a notion of arithmetic degree $\widehat{\deg}(\bar{D}^{d+1})=\widehat{\deg}_X(\bar{D}^{d+1})$.  The definition and properties of the arithmetic degree are discussed in sections 2 and 4 of \cite{M2}.  For example, as consequence of Proposition-Definition 2.4.3 and Proposition 4.5.4 in \cite{M2}, the arithmetic degree will satisfy the following properties:
\begin{enumerate}
\item[(1)] The degree $\widehat{\deg}((\lambda \bar{D})^{d+1})=\lambda^{d+1} \widehat{\deg}(\bar{D}^{d+1})$ for all $\lambda > 0$,
\item[(2)] the degree $\widehat{\deg}_{X'}((\varphi^* \bar{D})^{d+1})= \widehat{\deg}_{X}(\bar{D}^{d+1})$ for a birational morphism $\varphi \colon X' \lra X$ of normal, projective and geometrically integral varieties $X,X'$ over $K$.
\item[(3)] the degree $\widehat{\deg}((\bar{D} + \widehat{(f)})^{d+1})=\widehat{\deg}(\bar{D}^{d+1})$ for $f \in \Rat(X)_\R^\ast$.
\end{enumerate}

 For dynamical systems polarized by an integrable divisor $E$, we close this presentation with an arithmetic analogue of remark \ref{lambda to the n} for integrable divisors.

\begin{proposition} \label{arithmetic degree zero}  Let $K$ be a number field.  Let $\varphi \colon X \lra X$ be a birational morphism such that $(X,\varphi,E,\alpha)$ is a polarized dynamical system defined over $K$, with $\bar{E}$ integrable and $\dim(X)=d$.  Then the arithmetic degree $\widehat{\deg}((E^{d+1},\|.\|_\varphi))=\widehat{\deg}(\bar{E}^{d+1})=0$.
\end{proposition}
\begin{proof}  Let us denote $\bar{E}_\varphi=(E,\|.\|_\varphi)$.  Using the definition of polarized dynamics and the properties of the arithmetic degree we obtain
$$0=\widehat{\deg}((\varphi^* \bar{E})^{d+1}_\varphi)-\widehat{\deg}((\alpha \bar{E}_\varphi)^{d+1})=\widehat{\deg}(\bar{E}^{d+1}_\varphi)(1-\alpha^{d+1}),$$
and therefore $\widehat{\deg}(\bar{E}^{d+1}_\varphi)=0$.\end{proof}

\bibliography{HPol_bib.bib}{}

\begin{thebibliography}{10}

\bibitem{B1}
A.~Baragar.
\newblock Rational points on {K}3 surfaces in $\pl^1 \times \pl^1 \times
  \pl^1$.
\newblock {\em Math. Ann.}, 305:541--558, 1996.

\bibitem{B2}
A.~Baragar.
\newblock Canonical vector heights on {K}3 surfaces with picard number three -
  an argument for non-existence.
\newblock {\em Math. Comp.}, 73:2019--2025, 2004.

\bibitem{Be}
V.~G. Berkovich.
\newblock Spectral theory and analytic geometry over non-archimedean fields.
\newblock {\em Math. Surveys Monogr., Amer. Math. Soc.}, 33, 1990.

\bibitem{Bi}
H.~Billard.
\newblock Propri\'{e}t\'{e}s arithm\'{e}tiques d’une famille de surfaces
  {K}3.
\newblock {\em Compositio Math}, (3) 108:247--275, 1997.

\bibitem{B3}
J.~Burgos-Gil, A.~Moriwaki, P~Philippon, and M.~Sombra.
\newblock Arithmetic positivity in toric varieties.
\newblock {\em J. Algebraic Geometry}, 25:201--276, 2016.

\bibitem{Cha}
A~Chambert-Loir.
\newblock Mesures et \'equidistribution sur les espaces de berkovich.
\newblock {\em J. Reine Angew. Math.}, 595:215--235, 2006.

\bibitem{Chen1}
H.~Chen and A.~Moriwaki.
\newblock Algebraic dynamical systems and drichlet's unit theorem on arithmetic
  varieties.
\newblock {\em International Mathematics Research Notices}, 24:13669--13716,
  2015.

\bibitem{Chen2}
H.~Chen and A.~Moriwaki.
\newblock Sufficient conditions for the dirichlet property.
\newblock {\em hal-01502378}, 2017.

\bibitem{M1}
A.~Moriwaki.
\newblock Toward dirichlet's unit theorem on arithmetic varieties.
\newblock {\em Kyoto J. of Math.}, 53:197--259, 2013.

\bibitem{M2}
A.~Moriwaki.
\newblock Adelic divisors on arithmetic varieties.
\newblock {\em Memoirs of the American Mathematical Society}, 242(1144), 2016.

\bibitem{Sil-K3}
J.~Silverman.
\newblock Rational points on {K}3 surfaces: a new canonical height.
\newblock {\em Invent.~Math.}, 105:347--373, 1991.

\bibitem{W1}
L.~Wang.
\newblock Rational points and canonical heights on {K}3-surfaces in $\pl^1
  \times \pl^1 \times \pl^1$.
\newblock {\em Contemporary Math.}, 186:273--289, 1995.

\bibitem{We}
J.~Wehler.
\newblock {K}3-surfaces with picard number 2.
\newblock {\em Arch. Math.}, 50:73--78, 1988.

\bibitem{Zhang3}
S.~Zhang.
\newblock Small points and adelic metrics.
\newblock {\em Journal of Algebraic Geometry}, 4:281--300, 1995.

\bibitem{Zhang2}
S.~Zhang.
\newblock Distributions in algebraic dynamics. surveys in differential
  geometry.
\newblock {\em Int. Press, Somerville, MA.}, X:381--430, 2006.

\end{thebibliography}
\bibliographystyle{plain}

\end{document}